\documentclass[a4paper,reqno,11pt]{amsart}
\usepackage{amsmath}
\usepackage{amssymb}
\usepackage{cases}
\allowdisplaybreaks[4]

\newtheorem{thm}{Theorem}[section]
\newtheorem{prop}[thm]{Proposition}

\newtheorem{cor}[thm]{Corollary}

\newtheorem{conj}[thm]{Conjecture}

\theoremstyle{definition}
\newtheorem{definition}[thm]{Definition}
\newtheorem{defprop}[thm]{Definition-Proposition}

\newtheorem{example}[thm]{Example}

\theoremstyle{remark}

\numberwithin{equation}{section}

\newcommand{\bQ}{\mathbb{Q}}
\newcommand{\bP}{\mathbb{P}}

\newcommand\OO{{\mathcal{O}}}

\newcommand{\bC}{{\mathbb C}}

\newcommand\Sing{{\text{\rm Sing}}}
\newcommand\mult{{\text{\rm mult}}}

\newcommand\lct{{\text{\rm lct}}}
\newcommand\vol{{\text{vol}}}

\newcommand\qQ{{{\text{\rm q}}\mathbb{Q}}}
\newcommand\qW{{{\text{\rm qW}}}}

\begin{document}

\title{K-semistable Fano manifolds with the smallest alpha invariant}
\date{\today}
\author{Chen Jiang}
\address{Kavli IPMU (WPI), UTIAS, The University of Tokyo, Kashiwa, Chiba 277-8583, Japan.}
\email{chen.jiang@ipmu.jp}
\thanks{The author was supported by JSPS KAKENHI Grant Number JP16K17558 and World Premier International Research Center Initiative (WPI), MEXT, Japan.}

\maketitle

\begin{abstract}
In this short note, we show that K-semistable Fano manifolds with the smallest alpha invariant are projective spaces. Singular cases are also investigated.
\end{abstract}

\section{introduction}
Throughout the article, we work over the complex number field $\bC$.  A {\it $\bQ$-Fano variety} is a normal projective variety $X$ with log terminal singularities such that the anti-canonical divisor $-K_X$ is an ample Q-Cartier divisor.  It has been known that a {\it Fano manifold} $X$ (i.e., a smooth $\bQ$-Fano variety) admits
K\"ahler--Einstein metrics if and only if $X$ is {\it K-polystable} by the works
\cite{DT, Tia97, Don02, Don05, CT08, Sto09, Mab08, Mab09, Ber16} and
\cite{CDS15a, CDS15b, CDS15c, Tia15}. 

On the other hand, the existence of K\"ahler--Einstein metrics and K-stability are related to the {\it alpha invariants} $\alpha(X)$ of $X$ defined by Tian \cite{Tia87} (see also  \cite{TY87, Zel98, Lu00,Dem08}). Tian \cite{Tia87}  proved that for a Fano manifold $X$,  if $\alpha(X)>\dim X/(\dim X + 1)$, then $X$ admits K\"ahler--Einstein metrics. Odaka and Sano \cite[Theorem 1.4]{OS12} (see also its generalizations
\cite{Der14, BHJ15, FO, Fuj16b}) proved a variant of Tian's theorem: if a
$\bQ$-Fano variety $X$ satisfies that $\alpha(X) > \dim X/(\dim X + 1)$ (resp.
$\geq  \dim X/(\dim X + 1)$), then $X$ is K-stable (resp. K-semistable). We are interested in the relation of alpha invariants and K-semistability. 

Recall that Fujita and Odaka proved that there exists a lower bound of alpha invariants for K-semistable $\bQ$-Fano varieties.

\begin{thm}[{\cite[Theorem 3.5]{FO}}]
Let $X$ be a K-semistable $\bQ$-Fano variety of dimension $n$.

Then $\alpha(X)\geq \frac{1}{n+1}$.
\end{thm}

It is natural and interesting to ask when the equality holds. For example, it is well-known that $\bP^n$ is K-semistable with  $\alpha(\bP^n)= \frac{1}{n+1}$.
The main theorem of this paper is the following.
\begin{thm}\label{thm 1}
Let $X$ be a K-semistable Fano manifold of dimension $n$.

Then $\alpha(X)=\frac{1}{n+1}$ if and only if $X\cong \bP^n$.
\end{thm}

This is an application of Birkar's answer to Tian's question \cite[Theorem 1.5]{B}, and Fujita--Li's criterion for K-semistability \cite{Li15, F}.

It is natural to ask whether the same statement holds true for K-semistable $\bQ$-Fano varieties instead of manifolds. However, this is on longer true even in dimension $2$. We are grateful to Kento Fujita for kindly providing the following example:
\begin{example}Consider the cubic surface $X=(x_0^3=x_1x_2x_3)\subset \bP^3$, which is a toric {\it log del Pezzo surface} (i.e, a $\bQ$-Fano variety of dimension $2$) with $3$ du Val singularities of type $A_2$. On one hand, it is well-known that $X$ admits a K\"ahler--Einstein metric (cf. \cite{DT}), hence is K-semistable. On the other hand, $\alpha(X)=\frac{1}{3}$ (cf. \cite{PW}).

\end{example}
In fact, by the classification of possible du Val singularities of K-semistable log del Pezzo surfaces (cf. \cite[Corollary 6]{Liu}) and explicit computation of alpha invariants (cf. \cite{Park, PW, CK}), we have the following theorem.
\begin{thm}\label{thm dP}
Let $X$ be a K-semistable log del Pezzo surface with at worst du Val singularities.
Then $\alpha(X)=\frac{1}{3}$ if and only if $X\cong \bP^2$ or $X\subset \bP^3$ is a cubic surface with at least 2 singularities of type $A_2$. 

\end{thm}

Moreover, by classification of $\bQ$-Fano 3-fold with $\bQ$-factorial terminal singularities and $\rho(X)=1$ with large Fano index due to Prokhorov \cite{Pro10, Pro13}, we prove the following:

\begin{thm}\label{thm 3fold}
Let $X$ be a K-semistable $\bQ$-Fano 3-fold with $\bQ$-factorial terminal singularities and $\rho(X)=1$. Assume that $h^0(-K_X)\geq 22$.
Then $\alpha(X)=\frac{1}{4}$ if and only if $X\cong \bP^3$. 

\end{thm}

Finally, we propose the following much stronger conjecture. For some evidence in dimension $3$, we refer to \cite{CS} and \cite{Fuj15}.
\begin{conj}
Let $X$ be a K-semistable Fano manifold.

Then $\alpha(X)<\frac{1}{n}$ if and only if $X\cong \bP^n$.
\end{conj}

{\it Acknowledgments.} The author would like to thank Professors Kento Fujita and Yoshinori Gongyo for effective discussions. The main part of this paper was written during the author enjoyed the workshop ``Higher Dimensional Algebraic Geometry, Holomorphic Dynamics and Their Interactions" at Institute for Mathematical Sciences, National University of Singapore. The author is grateful for the hospitality and support of IMS.

\section{Preliminaries}
We adopt the standard notation and definitions in \cite{KM} and will freely use them.

\begin{definition}Let $X$ be a $\bQ$-Fano variety. The {\it alpha
invariant} $\alpha(X)$ of $X$ is defined by the supremum of positive rational
numbers $\alpha$ such that the pair $(X, \alpha D)$ is log canonical for any effective
$\bQ$-divisor $D$ with $D \sim_\bQ -K_X$. In other words, 
$$
\alpha(X):=\inf\{\lct(X; D)| 0\leq D \sim_\bQ -K_X\}.
$$
\end{definition}
Tian \cite{Tia90} asked whether whether the infimum is a minimum, which is answered by Birkar affirmatively.
\begin{thm}[{\cite[Theorem 1.5]{B}}]\label{Birkar answer}
Let $X$ be a $\bQ$-Fano variety. Assume that $\alpha(X)\leq 1$. Then there exists an effective $\bQ$-divisor $D$ such that $D \sim_\bQ -K_X$ and 
$\lct(X; D)=\alpha(X)$.
\end{thm} 
\begin{definition}[{\cite{F}}]Let $X$ be a $\bQ$-Fano variety of dimension $n$. Take any
projective birational morphism $\sigma : Y \to X$ with $Y$ normal and any
prime divisor $F$ on Y , that is, $F$ is a prime divisor {\it over} $X$.
\begin{enumerate}
\item Define the {\it log discrepancy} of $F$ as $A(F):=\mult_F (K_Y -\sigma^*K_X ) + 1$;
\item Define $\vol_X(-K_X-xF):=\vol_Y(-\sigma^*K_X-xF)$;
\item Define 
$$
\beta(F):=A(F)\cdot(-K_X)^n-\int_{0}^\infty \vol_X(-K_X-xF)\,\textrm{d}x.
$$
\end{enumerate}
\end{definition}
Note that the definitions do not depend on the choice of birational model $Y$.

Instead of recalling the original definition, we use the following criterion to define K-semistability.
\begin{defprop}[{\cite[Corollary 1.5]{F}, \cite[Theorem 3.7]{Li15}}]\label{defprop k-ss} Let $X$ be a $\bQ$-Fano variety.
$X$ is {\it K-semistable} if $\beta(F)\geq 0$ for any divisor $F$ over $X$.
\end{defprop}

Note that K-semistability is known to be equivalent to {\it Ding-semistability} by \cite{BBJ15}.
\section{Proof of main theorem}
\begin{prop}\label{prop main}
Let $X$ be a K-semistable $\bQ$-Fano variety of dimension $n$.
Assume that $\alpha(X)= \frac{1}{n+1}$, then there exists a prime divisor $E$ on $X$ such that $-K_X\sim_\bQ (n+1)E$ and $(X, E)$ is plt.
\end{prop}
\begin{proof}Let $X$ be a K-semistable $\bQ$-Fano variety of dimension $n$ with $\alpha(X)= \frac{1}{n+1}$.
By Theorem \ref{Birkar answer}, there is a divisor $D\sim_\bQ -K_X$ such that $\lct(X; D)=\frac{1}{n+1}$.
Take $F$ to be a non-klt place of $(X, \frac{1}{n+1}D)$, 
then there is a resolution $\sigma:Y\to X$ such that $F$ is a divisor on $Y$.

Denote $\mu$ to be the multiplicity of $F$ in $\sigma^*D$. Note that $\mu>0$ since $X$ is klt. By assumption, 
$$
\mult_F\bigg(K_Y-\sigma^*\bigg(K_X+\frac{1}{n+1}D\bigg)\bigg)=-1,
$$
which means that
$$A(F)=\frac{\mu}{n+1}.$$
By Definition-Proposition \ref{defprop k-ss}, $\beta(F)\geq 0$, which means that
\begin{align*}
\frac{1}{n+1}(-K_X)^n= {}& \frac{A(F)}{\mu}(-K_X)^n\\
\geq{}& \frac{1}{\mu}\int_0^\infty\vol_X(-K_X-xF) \,{\mathrm d}x\\
={}&\int_0^\infty\vol_X(-K_X-x\mu F)\,{\mathrm d}x\\
\geq{}&\int_0^\infty\vol_X(-K_X-xD)\,{\mathrm d}x\\
={}&\int_0^1 (1-x)^n(-K_X)^n\,{\mathrm d}x\\
={}&\frac{1}{n+1}(-K_X)^n.
\end{align*}
The second equality holds since $\sigma^*D\geq \mu F$.
Hence all inequalities become equalities. In particular, 
$$\vol_X(-K_X-x \mu F)=\vol_X(-K_X-xD)=(1-x)^n(-K_X)^n$$
for almost all $x$. 
By differentiability of volume functions (\cite[Corollary C]{BFJ09}),
\begin{align*}
{}&\mu \cdot \vol_{Y|F}(-\sigma^*K_X)\\
={}&-\frac{1}{n}\frac{\textrm{d}}{\textrm{d}x}\bigg|_{x=0}\vol_Y(-\sigma^*K_X-x\mu F)\\
={}&-\frac{1}{n}\frac{\textrm{d}}{\textrm{d}x}\bigg|_{x=0}(1-x)^n(-K_X)^n\\
={}&(-K_X)^n.
\end{align*}
Here $\vol_{Y|F}$ is the {\it restricted volume}, we refer to \cite{ELMNP09} for definition and properties.
Since $\vol_{Y|F}(-\sigma^*K_X)>0,$ $F\not \subseteq \mathbf{B}_+(-\sigma^*K_X)$ by \cite[Theorem C]{ELMNP09}.
Hence by \cite[Corollary 2.17]{ELMNP09}, 
$$\vol_{Y|F}(-\sigma^*K_X)=(-\sigma^*K_X)^{n-1}\cdot F=(-K_X)^{n-1}\cdot \sigma_*F.$$
In other words, we have
$$
(-K_X)^{n-1}(D-\mu\sigma_*F)=(-K_X)^n-\mu \cdot \vol_{Y|F}(-\sigma^*K_X)=0.
$$
This implies that  $D=\mu \sigma_* F$ since $D\geq \mu  \sigma_*  F$ and $-K_X$ is ample. In particular, $F$ is not $\sigma$-exceptional and $\sigma_*F$ is a prime divisor on $X$. Denote $E:=\sigma_*F$. Moreover, since $F$ is a non-klt place of $(X, \frac{1}{n+1}D)$, $\mult_E \frac{1}{n+1}D=1$, that is, $\mu=n+1$. In particular, $-K_X\sim_\bQ D=(n+1)E$. 
Finally, this argument shows that $F$ is the only non-klt place of $(X, E)$, which means that $(X, E)$ is plt.
\end{proof}

\begin{cor}\label{cor 1}
Let $(X, E)$ as in Proposition \ref{prop main}. Then $X\simeq \bP^n$ if one of the following condition holds:
\begin{enumerate}
\item $X$ is factorial;
\item $(E)^n\geq 1$;
\item $E$ is Cartier in codimension two and $E\simeq \bP^{n-1}$. 
\end{enumerate}
\end{cor}
\begin{proof}
(1) If $X$ is factorial, then $E$ is a Cartier divisor. In particular, $(E)^n\geq 1$. Hence this is a special case of (2).

(2) If $(E)^n\geq 1$, then
$$
(-K_X)^n=(n+1)^n(E)^n\geq (n+1)^n. 
$$
By \cite[Theorem 1.1]{Liu} or \cite[Theorem 9]{LZ16}, $X\simeq \bP^n$.

(3) If $E$ is Cartier in codimension two and $E\simeq \bP^{n-1}$, then by adjunction, $(K_X+E)|_E=K_E$, and
$$
(-K_X)^n=\frac{(n+1)^n}{n^{n-1}}(-(K_X+E))^{n-1}\cdot E=\frac{(n+1)^n}{n^{n-1}}(-K_E)^{n-1}=(n+1)^n.
$$
Again by \cite[Theorem 1.1]{Liu} or \cite[Theorem 9]{LZ16}, $X\simeq \bP^n$.
\end{proof}

\begin{proof}[Proof of Theorem \ref{thm 1}] It follows directly from Proposition \ref{prop main} and Corollary \ref{cor 1}(1) (or \cite{KO}).
\end{proof}

\section{Singular surfaces}
Recall the following theorem on classification of possible du Val singularities of a K-semistable log del Pezzo surface.
\begin{thm}[{\cite[Theorem 23, Proof of Corollary 6]{Liu}}]\label{dp 1} Let $X$ be a K-semistable log del Pezzo surface with at worst du Val
singularities. 
\begin{enumerate}
\item If $(-K_X)^2= 1$, then $X$ has at worst singularities of type $A_1$, $A_2$, $A_3$, $A_4$, $A_5$,
$A_6$, $A_7$, $A_8$, or $D_4$;
\item If $(-K_X)^2= 2$, then $X$ has at worst singularities of type $A_1$, $A_2$, or $A_3$;

\item If $(-K_X)^2= 3$, then $X$ has at worst singularities of type $A_1$ or $A_2$;
\item If $(-K_X)^2= 4$, then $X$ has at worst singularities of type $A_1$;
\item If $(-K_X)^2\geq 5$, then $X$ is smooth.

\end{enumerate}
\end{thm}

We remark that in \cite[Corollary 6]{Liu}, log del Pezzo surfaces are assumed to be admitting K\"ahler--Einstein metrics, but the proof works well for K-semistable log del Pezzo surfaces. The only part that the existence of K\"ahler--Einstein metrics is needed is to exclude the case that $(-K_X)^2=1$ and $X$ has singularities of type $A_8$.

Recall the following theorem on explicit computation of alpha invariants.
\begin{thm}[{\cite{Park}, \cite[Theorems 1.4, 1.5, and 1.6]{PW}, \cite[Theorem 1.26, Example 1.27]{CK}}]\label{dp 2}
Let $X$ be a log del Pezzo surface with at worst du Val
singularities. Assume that $X$ is singular, then $\alpha(X)=\frac{1}{3}$ if and only if one of the following holds:
\begin{enumerate}
\item $(-K_X)^2= 6$ and $\Sing(X)=\{A_1\}$;
\item $(-K_X)^2= 5$ and $\Sing(X)=\{A_2\}$ or $\{2A_1\}$;
\item $(-K_X)^2= 4$ and $\Sing(X)=\{A_3\}$ or $\Sing(X)\supseteq\{A_1+A_2\}$;
\item $(-K_X)^2= 3$ and $\Sing(X)\supseteq \{A_4\}$, $\{2A_2\}$, or $\Sing(X)=\{D_4\}$;
\item $(-K_X)^2= 2$ and $\Sing(X)\supseteq \{D_5\}, \{(A_5)'\}$, or $\{A_7\}$;
\item $(-K_X)^2= 1$ and $\Sing(X)\supseteq \{D_8\}$ or $\{E_6\}$.
\end{enumerate}

\end{thm}

\begin{proof}[Proof of Theorem \ref{thm dP}] Let $X$ be a K-semistable log del Pezzo surface with at worst du Val singularities and $\alpha(X)=\frac{1}{3}$. If $X$ is smooth, then $X\simeq \bP^2$ by Theorem \ref{thm 1}. If $X$ is singular, then $(-K_X)^2=3$ and $\Sing(X)\supseteq \{2A_2\}$ by Theorems \ref{dp 1} and \ref{dp 2}. To see the ``if" part, one just notice that any cubic surface with at worst singularities of type $A_1$ or $A_2$ is K-semistable (cf. \cite[Theorem 4.3]{OSS16}).
\end{proof}

\section{Singular threefolds}
In this section, we prove Theorem \ref{thm 3fold}. Recall the following theorem on the upper bound of volumes.
\begin{thm}[{cf. \cite[Theorem 25]{Liu}}]\label{terminal volume}
Let $X$ be a K-semistable $\bQ$-Fano $3$-fold with at worst terminal singularities. Let $p\in X$ be an isolated
singularity with local index $r$.  Then
$$
(-K_X)^3\leq \frac{(r+2)(4+4r)^3}{(3r)^3}.
$$
\end{thm}

\begin{proof}
Denote by $\mathfrak{m}_p$ the maximal ideal at $p$. We may take a log resolution of $(X, \mathfrak{m}_p)$, namely $\pi: Y \to X$
such that $\pi$ is an isomorphism away from $p$ and $\pi^{-1}\mathfrak{m}_p \cdot \OO_Y$ is an invertible ideal sheaf on $Y$. Let $E_i$ be exceptional divisors of $\pi$. We define the numbers $a_i$ and $b_i$ by
$$K_Y = \pi^* K_X+\sum_i a_iE_i$$
 and 
 $$\pi^{-1}\mathfrak{m}_p \cdot \OO_Y= \OO_Y(-\sum_i b_iE_i).$$
It is clear that $\lct(X; \mathfrak{m}_p) = \min_i\frac{1+a_i}{b_i}$. Since $\pi$ is an isomorphism away from $p$, we have
$b_i\geq 1$ for any $i$. Since $X$ is terminal at $p$, by \cite{Kaw93}, there exists an index $i_0$ such that $a_{i_0}=\frac{1}{r}.$
Hence
$$\lct(X; \mathfrak{m}_p) \leq \frac{1 + a_{i_0}}{b_{i_0}}\leq 1+\frac{1}{r}.$$

On the other hand, by \cite{Kakimi} (see also \cite[Proposition 3.10]{TW}), $\mult_pX\leq r+2$.
Hence by \cite[Theorem 16]{Liu}, 
$$
(-K_X)^3\leq \bigg(1+\frac{1}{3}\bigg)^3\lct(X; {\mathfrak{m}}_p)^3\mult_pX \leq \frac{(r+2)(4+4r)^3}{(3r)^3}.
$$
\end{proof}

Now let $X$ be a K-semistable $\bQ$-Fano 3-fold with $\bQ$-factorial terminal singularities and $\rho(X)=1$ with $\alpha(X)=\frac{1}{4}$. By Proposition \ref{prop main}, there exists a prime divisor $E$ on $X$ such that $-K_X\sim_\bQ 4E$.

Recall that we may define (\cite{Pro10})
\begin{align*}
\qW(X) {}&:=\max\{q\mid -K_X\sim qA, A\text{ is a Weil divisor}\},\\
\qQ(X) {}&:=\max\{q\mid -K_X\sim_\bQ qA, A\text{ is a Weil divisor}\}.
\end{align*}
It is known by \cite{Suz04, Pro10} that 
$$
\qW(X), \qQ(X)\in \{1,\ldots, 11,13,17,19\}.
$$
Moreover, by \cite[Lemma 3.2]{Pro10}, in our case, $4|\qQ(X)$.  Hence there are 2 cases: (i) $\qQ(X)=8$; (ii) $\qQ(X)=4$.

Now assume that  $h^0(-K_X)\geq 22$. Define the genus $g(X):=h^0(-K_X)-2\geq 20$. 

If $\qQ(X)=8$, since $g(X)>10$, then by \cite[Theorem 1.2(ii)]{Pro13}, either $X\simeq X_6 \subset \bP(1, 2, 3,3, 5)$ or  $X\simeq  X_{10} \subset \bP(1,2,3,5,7)$. But in either case, $-K_X\sim 8A$ where $A$ is an effective divisor, which implies that $\alpha(X)\leq \frac{1}{8}$ since $(X, A)$ is not klt, a contradiction.

Now assume that $\qQ(X)=4$, by \cite[Lemma 8.3]{Pro13}, $\text{Cl}(X)$ is torsion-free and $\qW(X)=\qQ(X)$, hence there is a Weil divisor $A$ such that $-K_X\sim 4A$. If $g(X)\geq 22$, then by \cite[Theorem 1.2(vi)]{Pro13}, $X\simeq \bP^3$ or $X_4\subset \bP(1,1,1,2,3)$. The latter is absurd, since it has a singularity of index $3$, and $(-K_{X_4})^3=128/3$, which contradicts to Theorem \ref{terminal volume}. If $20\leq g(X)\leq 21$, then we have the following possibilities due to computer computation (see \cite{GRD}, or \cite{BS07, Pro10, Pro13}):
$$ \begin{array}{ccc}
g(X) & \mathbf{B} & A^3 \\
\hline
21&\{3\}&2/3\\
20&\{5,7\}&22/35
\end{array} $$
Here $\mathbf{B}$ is the set local indices of singular points.
It is easy to see that both cases contradict to Theorem \ref{terminal volume}.



In summary, Theorem \ref{thm 3fold} is proved.

\end{document}